\documentclass[11pt, oneside]{amsart}

\usepackage{tabularx, hyperref}
\usepackage{amssymb} \usepackage{amsfonts} \usepackage{amsmath}
\usepackage{amsthm} \usepackage{epsfig, subfig}
\usepackage{ amscd, amsxtra, latexsym}
\usepackage[all]{xy}
\usepackage{caption}
\usepackage{enumerate}
\usepackage{color}

\newtheorem{lemma}{Lemma}[section]
\newtheorem{thm}[lemma]{Theorem}
\newtheorem{prop}[lemma]{Proposition}
\newtheorem{cor}[lemma]{Corollary}

\theoremstyle{definition}
\newtheorem{defn}[lemma]{Definition}

\newtheorem{rem}[lemma]{Remark}
\newtheorem{conv}[lemma]{Convention}
\theoremstyle{definition}

\long\def\comment#1\endcomment{}

\newcommand{\N}{\ensuremath {\mathbb{N}}}
\newcommand{\R} {\ensuremath {\mathbb{R}}}

\newcommand{\matP} {\ensuremath {\mathbb{P}}}

\newcommand{\matE} {\ensuremath {\mathbb{E}}}

\newcommand{\calC} {\ensuremath {\mathcal{C}}}

\newcommand{\calP} {\ensuremath {\mathcal{P}}}

\newcommand{\calM} {\ensuremath {\mathcal{M}}}
\newcommand{\calH} {\ensuremath {\mathcal{H}}}

\newcommand{\calZ} {\ensuremath {\mathcal{Z}}}

\definecolor{darkgreen}{cmyk}{1,0,1,.2}

\author{Alessandro Sisto}
\address{Mathematical Institute, 24-29 St Giles', Oxford OX1 3LB, United Kingdom}
\email{sisto@maths.ox.ac.uk}
\title{Tracking rates of random walks}
\thanks{The author was funded by the EPSRC grant "Geometric and analytic aspects of infinite groups".}

\begin{document}
\begin{abstract}
 We show that simple random walks on (non-trivial) relatively hyperbolic groups stay $O(\log(n))$-close to geodesics, where $n$ is the number of steps of the walk. Using similar techniques we show that simple random walks in mapping class groups stay $O(\sqrt{n\log(n)})$-close to geodesics and hierarchy paths. Along the way, we also prove a refinement of the result that mapping class groups have quadratic divergence.

An application of our theorem for relatively hyperbolic groups is that random triangles in non-trivial relatively hyperbolic groups are $O(\log(n))$-thin, random points have $O(\log(n))$-small Gromov product and that in many cases the average Dehn function is subasymptotic to the Dehn function.
\end{abstract}

\maketitle

\section{Introduction}
It is known in several contexts that sample paths of random walks stay sublinearly close to geodesics \cite{FurKen-trkglnr,Ka-sym,Ka-entr,Ka-hyp,KaMa-unconv,Ti}. Such a property is useful, for example, to describe the Poisson boundary \cite{Ka-entr}.

It seems that little is known about estimates on the tracking rates, i.e. the actual expected value of the Hausdorff distance between a random path and a corresponding geodesic. The tracking rate in non-abelian free groups \cite{Led} and more generally non-elementary hyperbolic groups \cite[Corollary 3.9]{BHM-meas} is logarithmic in the length of the walk. Our first main result is that the logarithmic rate holds for a more general class of groups, and we will show it with entirely different, more geometric, methods than \cite{Led} and \cite{BHM-meas}. We say that a relatively hyperbolic group is non-trivial if it is not virtually cyclic and all peripheral subgroups have infinite index. We denote the Hausdorff distance by $d_{Haus}$.

\begin{thm}
\label{logtrkintro}
Let $X_n$ be a simple random walk on the non-trivial relatively hyperbolic group $G$. There exists $C$ so that for each $n\geq 2$ we have
$$\mathbb{E}\left[\sup_{[1,X_n]} d_{Haus}(\{X_i\}_{i\leq n}, [1,X_n])\right]\leq C\log (n),$$
where the supremum is taken over all geodesics $[1,X_n]$ from $1$ to $X_n$.
\end{thm}

Notice in particular that the expected Hausdorff distance between two geodesics from $1$ to $X_n$ is at most logarithmic.

Sublinear tracking has been shown in \cite{Ka-hyp} for hyperbolic groups and very recently in \cite{Ti} for relatively hyperbolic groups (in both cases for random walks of finite first moment).

We will actually show a more general result, Theorem \ref{logtrk}, which allows more general random walks, gives a polynomial decay of the probability that a sample path gives an ``off-range'' Hausdorff distance and deals with group actions on relatively hyperbolic spaces instead of relatively hyperbolic groups. The motivation for looking at such actions is on one hand that we will apply Theorem \ref{logtrk} to the action of a mapping class group on the corresponding curve complex, and on the other that such group actions are very much related to the notion of hyperbolically embedded subgroups as defined in \cite{DGO}, see \cite[Theorem 4.42]{DGO} and \cite[Theorem 6.4]{Si-metrrh}. Namely, any hyperbolically embedded subgroup gives an action of the ambient group on a relatively hyperbolic space (a Cayley graph with respect to a possibly infinite generating system), and viceversa a nice action on a relatively hyperbolic space gives hyperbolically embedded subgroups.

Our next result is that sublinear tracking holds in mapping class groups as well. Recall that the complexity of a surface of finite type is $3g+p-3$ where $g$ is the genus of the surface and $p$ the number of punctures.

\begin{thm}
Let $S$ be a connected, orientable surface $S$ of finite type, with empty boundary and complexity at least 2.
 Let $\calM(S)$ be its mapping class group and let $\{X_n\}$ be a simple random walk on $\calM(S)$. Then
$$\matE\left[\sup_{\gamma(X_n)} d_{Haus}(\{X_i\}_{i\leq n}, \gamma(X_n))\right] =O(\sqrt{n\log(n)}),$$
where the supremum is taken over all geodesics in a given word metric and hierarchy paths $\gamma(X_n)$ from $1$ to $X_n$.
\end{thm}

Once again, given any choice of a pair of hierarchy paths or geodesics from $1$ to $X_n$ the expected value of their Hausdorff distance is $O(\sqrt{n\log(n)})$.

We remark that not even sublinear tracking seems to appear in the literature. Once again, we will show a stronger result (Theorem \ref{trkmcg}). It is quite possible that the same techniques we will use to show the theorem apply in other contexts as well. In fact, we will use machinery (most notably the Distance Formula) that is currently available for mapping class groups only but should have analogues for other groups. Indeed, several results about the geometry of mapping class groups have been very recently extended to right-angled Artin groups in \cite{KK-raagvsmcg} (but an improved version of the Distance Formula contained in that paper would be needed for our proofs to carry over).

In order to prove the theorem we will need a refinement, which may be of independent interest, of the result that mapping class groups have (at least) quadratic divergence \cite{Be-th, DR-divteich}. (Recall that the divergence is, roughly speaking, the minimal length of paths avoiding a ball as a function of the diameter of the ball, but we will not need the exact definition. The interested reader is referred to \cite{DMS-div}.)

The $D$-bounded pairs appearing in the Proposition below will be defined in Subsection \ref{divsec}, for the moment we will just mention that all pairs of points on the orbit of a pseudo-Anosov $g$ are $D$-bounded, but $D$ depends on the choice of $g$.

\begin{prop}
\label{prop:quadr}
Let $S$ be a connected, orientable surface $S$ of finite type, with empty boundary and complexity at least 2.
 Let $\calM(S)$ be its mapping class group. Then there is a constant $C=C(S)$ so that the following holds. Let $D\geq 1$ and let $x_1,x_2\in \calM(S)$ be a $D$-bounded pair. Then for any path $\alpha$ of length at least $1$ from $x_1$ to $x_2$ and any $p\in [x_1,x_2]$ we have
$$d(p,\alpha)\leq C \sqrt{Dl(\alpha)},$$
where $[x_1,x_2]$ can denote either a hierarchy path or a geodesic in a given word metric from $x_1$ to $x_2$.
\end{prop}

What is shown in \cite{Be-th, DR-divteich} is that for each $D$ there exists $K(D)$ so that $d(p,\alpha)\leq K(D) \sqrt{l(\alpha)}$, so the improvement brought by Proposition \ref{prop:quadr} is to show that $K(D)$ can be chosen to be linear in $\sqrt{D}$. This result can also presumably be obtained using the techniques in \cite[Section 6]{Be-th}. However, our proof, which is inspired by arguments in \cite{KaLe-grphquad, DMS-div}, is different and shorter. The $C\sqrt{D}$ coefficient in front of $\sqrt{l(\alpha)}$ should be optimal in the sense that it should not be possible to replace it by any function in $o(\sqrt{D})$.

Finally, we will give two applications of Theorem \ref{logtrk}. We say that random triangles in a given group are $\omega$-thin, where $\omega:\R^+\to\R^+$, if the expected thinness constant of triangles joining the endpoints of three independent simple random walks is $O(\omega(n))$, where $n$ is the number of steps of both walks.
Also, we say that that random points have $\omega$-small Gromov product if all geodesics joining the endpoints of two independent random walks pass $O(\omega)$-close to $1$.
Finally, given a combing of a group, we define the average Dehn function to be the expected area of loops obtained concatenating the sample path of a simple random walk and the path from the given combing joining the endpoint to the identity. All concepts will be formally defined and discussed in Section \ref{appl}.

\begin{thm}
 Let $G$ be a non-trivial relatively hyperbolic group. Then
\begin{enumerate}
 \item random triangles in $G$ are $\log(n)$-thin.
 \item random points in $G$ have $\log(n)$-small Gromov product.
 \item if all Dehn functions of the peripheral subgroups are bounded by $\delta:\R^+\to\R^+$ and are at most polynomial then for any geodesic combing the average Dehn function is $O(\, n\, \delta(\log(n))\, )$.
\end{enumerate}
\end{thm}

In particular, we see that expected values of both the thinness constant and the Dehn function can be much lower than worst-case values.
Also, notice that if all triangles in a group are $\log(n)$-thin then the group is hyperbolic (see for example \cite{Gro-hyp}, \cite{Dr-surv} or \cite{FS}) and that if the Dehn function is subquadratic then it is actually linear and the group is hyperbolic \cite{Ols-subquad} (see also \cite[pages 149-157]{ggt-proc}).
In view of this, the theorem can be interpreted as indicating that random configurations in relatively hyperbolic groups resemble corresponding configurations in hyperbolic groups.

\subsection*{Acknowledgement} The author would like to thank Ilya Kapovich for the suggestions that kick-started this work, Fran\c{c}ois Ledrappier for helpful comments and Cornelia Dru\c{t}u for numerous useful comments, corrections and suggestions on preliminary drafts of this paper.

\section{Outline}
We emphasise that our methods of proof are completely different from other methods that have been used to show sublinear tracking. The proof of Theorem \ref{logtrkintro} uses three main ingredients, discussed in the following subsections.

\subsection{Projections estimate} We will consider closest point projections on peripheral sets and show that it is unlikely that two random points project far away on some peripheral set (left coset of peripheral subgroup in the context of groups). The main tool we will use to show this is Corollary \ref{B-type}, an inequality for closest point projections on peripheral sets, pointed out in \cite{Si-proj}, which is similar to a very useful inequality due to Behrstock \cite{Be-th} in the context of subsurface projections. This step can be skipped if one wants to show Theorem \ref{logtrkintro} for hyperbolic groups only\footnote{An abridged version of the argument for hyperbolic groups is available on the author's blog \href{http://alexsisto.wordpress.com/2013/01/28/tracking-of-random-walks-with-geodesics/}{http://alexsisto.wordpress.com/2013/01/28/tracking-of-random-walks-with-geodesics/}}. The outcomes of this argument are Lemma \ref{smallproj} and Lemma \ref{smallrndprj}. We record the latter as it could be useful in other contexts as well.

\subsection{Exponential divergence} The divergence of non-trivial relatively hyperbolic groups is at least exponential (see \cite{Si-metrrh}). The same proof, using the consequence of the projections estimate, guarantees that points on a geodesic connecting the endpoints of a random path are close to the random path. For hyperbolic groups, this is a standard argument, see \cite[Proposition III.H.1.6]{BrHae}.

\subsection{Drift estimates} The last ingredient is a fact about random walks on non-amenable groups, namely the fact that random walks on nonamenable groups make linear progress (see \cite[Lemma 8.1(b)]{Wo}). This allows us to exclude the existence of ``large detours'' in the random path.
\par\medskip
Except for exponential divergence, all ingredients are available for mapping class groups. In that context instead of exponential divergence we have quadratic divergence, and this is why the rate we get for mapping class groups is worse than the one for relatively hyperbolic groups. We will actually proceed slightly differently in the mapping class group case in order to exhibit a variation on the argument for relatively hyperbolic groups, and we will use a drift estimate in the curve complex due to Maher \cite{Ma-linprog}. 

\section{Relatively hyperbolic groups}
In this section we recall the results about relatively hyperbolic spaces that we will need. We remark that relatively hyperbolic spaces are also called asymptotically tree-graded spaces in, e.g., \cite{DS-treegr,D-relhyp,Si-proj}. The notion of metric relative hyperbolicity coincides in the context of Cayley graphs of groups with the notion of (strong) relative hyperbolicity as studied, for example, in \cite{Gro-hyp,Fa-relhyp,Bow-99-rel-hyp,Os-rh, GrMa-perfill}. Throughout the section, let $X$ be a geodesic metric space hyperbolic relative to the collection of subsets $\calP$, called peripheral sets.

The following lemma can be found in \cite{Si-proj}. It also follows by combining facts about relative hyperbolicity discovered by Dru\c{t}u and Sapir \cite{DS-treegr}.

\begin{lemma}
For each $P\in\calP$ denote by $\pi_P:X\to P$ a coarse closest point projection, i.e. a function so that $d(x,\pi_P(x))\leq d(x,P)+1$. There exists $C$ with the following properties.
\begin{enumerate}
 \item For each distinct $P,Q\in\calP$ we have $diam(\pi_P(Q))\leq C$.
 \item For each $x,y\in X$ and $P\in\calP$ so that $d(\pi_P(x),\pi_P(y))\geq C$ we have $d([x,y],\pi_P(x)),\, d([x,y],\pi_P(y))\leq C$ for every geodesic $[x,y]$.
\end{enumerate}
\end{lemma}

In other words, the projection of one peripheral set onto another one has bounded diameter and if two points project far away on some peripheral set then the geodesic connecting them passes close to the projection points.

The following corollary follows from the lemma by standard arguments, see e.g. \cite[Lemma 2.5]{Si-contr}.

\begin{cor}[Projections estimate, cfr. Theorem \ref{Behin}]
\label{B-type}
There exists $B$ with the following property. Let $P,Q\in\calP$ be distinct and let $x\in X$. Then
$$\min\{d(\pi_P(x),\pi_P(Q)),\, d(\pi_Q(x),\pi_Q(P))\}\leq B.$$
\end{cor}

So, if $x$ and $Q$ have far away projections on $P$, then $x$ and $P$ have close projections on $Q$, providing a useful trick to control a projection. 

In \cite{Si-metrrh} relative hyperbolicity has been characterised in terms of transient sets of geodesics, that were introduced in \cite{Hr-relqconv}. Roughly speaking, a point on a geodesic fails to be transient if it is well-within a subgeodesic that fellow-travels a peripheral set. The formal definition is below.

Let $\mu,R$ be constants and $\alpha$ a geodesic in $X$. Denote by $deep_{\mu,R}(\alpha)$ the set of points $p$ of $\alpha$ that belong to some subgeodesic $[x,y]$ of $\alpha$ with endpoints in $N_{\mu}(P)$ for some $P\in\calP$ and so that $d(p,x),d(p,y)>R$. Denote $trans_{\mu,R}(\alpha)=\alpha\backslash deep_{\mu,R}(\alpha)$, the set of \emph{transient points}. 

The reader is referred to \cite{Hr-relqconv,Si-metrrh} for the following properties of transient and deep sets. Some of them follow from results in \cite{DS-treegr} which are however not phrased in terms of these notions.

\begin{lemma}
\label{transprop}
 There exist $\mu,R,D,t,C$ with the following properties.
\begin{enumerate}
 \item $[$Relative Rips condition$]$ For each $x,y,z\in X$ we have
$$trans_{\mu,R}([x,y])\subseteq N_{D}(trans_{\mu,R}([x,z])\cup trans_{\mu,R}([z,y])).$$
 \item $deep_{\mu,R}([x,y])$ is contained in a disjoint union of subgeodesics of $[x,y]$ each contained in $N_{t\mu}(P)$ for some $P\in\calP$, called \emph{deep component along} $P$.
 \item The endpoints of the deep component of $[x,y]$ along $P\in\calP$ (if it exists) are $C$-close to $\pi_P(x),\pi_P(y)$.
 \item If for some $P\in\calP$ we have $d(\pi_P(x),\pi_P(y))>C$, then $[x,y]$ has a deep component along $P$ of length at least $d(\pi_P(x),\pi_P(y))-C$.
\end{enumerate}
\end{lemma}

\begin{conv}
 When we write $trans$ instead of $trans_{\mu,R}$ we implicitly fix constants $\mu,R$ as in the lemma.
\end{conv}

\section{Logarithmic tracking}

The aim of this section is to show the following. We denote a supremum over all geodesics from $x$ to $y$ by $\sup_{[x,y]}$. Recall that we say that a relatively hyperbolic group is non-trivial if it is not virtually cyclic and all peripheral subgroups have infinite index.

\begin{thm}
Let $\{X_n\}$ be a simple random walk on the non-trivial relatively hyperbolic group $G$. There exists $C$ so that for each $n\geq 2$ we have
$$\mathbb{E}\left[\sup_{[1,X_n]} d_{Haus}(\{X_i\}_{i\leq n}, [1,X_n])\right]\leq C\log (n).$$
\end{thm}

We will actually show the following refinement.
Following, e.g., \cite{Ma-linprog}, we say that a random walk $\{X_n\}$ on the group $G$ acting on the pointed metric space $(X,p)$ makes \emph{linear progress} if there exists $C_0\geq 1$ so that
$$\matP\left[d(p,X_np)\leq n/C_0\right]\leq C_0 e^{-n/C_0}.$$
As noticed in \cite[Proposition 5.9]{CM-statscl}, it follows from \cite[Lemma 8.1(b)]{Wo} that when $G$ is a non-amenable group acting on itself, any symmetric random walk makes linear progress.

\begin{thm}
\label{logtrk}
Suppose that the finitely generated group $G$ acts by isometries on the relatively hyperbolic space $X$ permuting the peripheral sets. Suppose also that there are at least 2 peripheral sets and that the stabiliser of each peripheral set has unbounded orbits. Let $\mu$ be a symmetric probability measure on $G$ whose finite support generates $G$ and let $\{X_n\}$ be the corresponding random walk, which we assume to make linear progress.
Then for each $p\in X$ and for each $k\geq 1$ there exists $C$ so that
$$\mathbb{P}\left[\sup_{[p,X_np]} d_{Haus}(\{X_ip\}_{i\leq n}, trans([p,X_np]))\geq C\log (n)\right]\leq Cn^{-k}.$$
Moreover, the same is true for $[p,X_np]$ substituting $trans([p,X_np])$.
\end{thm}

Let us fix the notation of the theorem, including $k$. Let $\calH$ be the collection of peripheral sets of $X$. We now state and prove the three lemmas we need, and we will combine them together in the next subsection.

The following general fact about relative hyperbolicity will immediately imply that each point on $trans([p,X_np])$ is close to $\{X_ip\}_{i\leq n}$.

\begin{lemma}
\label{div}
 There exists $C_1$ with the following property. Let $\alpha$ be a discrete path from $x$ to $y$, where $x,y\in X$, and for some geodesic $[x,y]$ let $p$ be a point on $ trans([x,y])$. Then $d(p,\alpha)\leq C_1\log_2(l(\alpha)+1)+C_1$.
\end{lemma}

\noindent\emph{Proof.}
 The proof is the same as \cite[Proposition 6.17]{Si-metrrh} and is an easy generalization of, e.g., \cite[Proposition III.H.1.6]{BrHae}. We give the proof for the sake of completeness.

We argue by induction on $k$ such that $\mathrm{length}\, (\alpha ) \leq 2^k$.
 
Let $D$ be as in Lemma \ref{transprop}. If $l(\alpha)\leq 2$, then the lemma holds, with $C_1=2$. Assume that the statement is proven for paths of length $\leq 2^k$, let $\alpha$ be a path of length $\leq 2^{k+1}$. Split $\alpha$ into paths $\alpha_i$ of length $l(\alpha)/2\geq 1$ and let $q$ be the common endpoint. Then $p$ is $D$-close to some $p'\in trans([x,q])\cup trans([q,y])$. By induction we have that $d(p',\alpha)\leq D\log_2(l(\alpha)/2)+2$, so that
\[
 \pushQED{\qed}
d(p,\alpha)\leq d(p,p')+d(p',\alpha)\leq D\log_2(l(\alpha))+2.\qedhere
 \popQED
\]

\begin{figure}[ht]
 \includegraphics[scale=0.7]{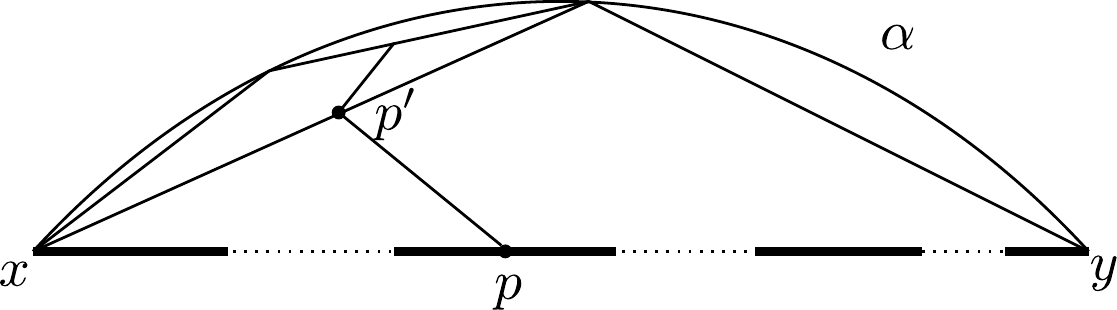}
\caption{Proof of Lemma \ref{div}. The thick segments along $[x,y]$ represent the transient set.}
\end{figure}

Let us now show that the deep components of $[p,X_np]$ are expected to be logarithmically small. It will be convenient to set, for $H\in\calH$,
$$d_H(\cdot,\cdot)=d(\pi_H(\cdot),\pi_H(\cdot)),$$
as it is customary for subsurface projections.

\begin{lemma}
\label{smallproj}
 There exists $C_2$ so that, for each $n\geq 1$,
 $$\matP\left[\exists H\in\calH: d_H(p,X_np)\geq C_2\log (n)\right]\leq C_2n^{-k}.$$
\end{lemma}

\begin{proof}
The usual notation $\matP[\cdot|\cdot]$ will be used for the conditional probability. We will show that there exists $K$ so that:
\begin{enumerate}
 \item for all $l\geq 0$ and $H\in\calH$ we have
  $$\matP\left[d_H(p,X_np)\geq l\right|d_{H}(p,X_np)\geq K]\leq Ke^{-l/K}.$$
  \item for any $g\in G$ the set $A(g)=\{H\in \calH: d_H(p,gp)\geq K\}$ satisfies $|A(x)|\leq K d(p,gp)$. 
\end{enumerate}

Using these two facts we can make the estimate:
\begin{flushleft}
$\matP\left[\exists H\in\calH: d_H(p,X_np)\geq C_2\log (n)\right]\leq$
\end{flushleft}
\begin{flushright}
 $ \sum_{H\in\calH}\matP\big[d_H(p,X_np)\geq C_2\log (n)|H\in A(X_n)\big]\matP[H\in A(X_n)]\leq$
\end{flushright}
\begin{flushright}
 $(Ke^{-C_2\log(n)/K})\sum_{H\in\calH}\matP[H\in A(X_n)]\leq (KTn) (Ke^{-C_2\log(n)/K}),$
\end{flushright}

where $T=\max\{d(p,gp):g\in supp(\mu)\}$. The last inequality follows from the observation that the random variable $|A(\cdot)|$ is the sum of the indicator functions $1_{H\in A(\cdot)}$, so that 
$$\sum_{H\in\calH} \matP[H\in A(X_n)]=\matE[|A(X_n)|].$$
We can then clearly choose $C_2$ large enough that is satisfies the lemma.
\par
$(1)$ The proof is similar to that of \cite[Lemma 6.2]{Si-contr}.
We want to show that there exists $K'$ so that 
$$\matP\left[d_H(p,X_np)\geq l+K'\right]\leq K' \matP\left[d_H(p,X_np)\in [l-K',l+K')\right],$$
which then implies the exponential decay we are looking for (just as in \cite[Lemma 6.2]{Si-contr}). Namely, one readily shows inductively
$$\matP\left[d_H(p,X_np)\geq (2i+1)K'\right]\leq (1+1/K')^{-i}\matP[d_H(p,X_np)\geq K'],$$
just using the inequality above for $l=(2i+2)K'$:
\begin{flushleft}
 $\matP\left[d_H(p,X_np)\geq (2i+3)K'\right]\leq$
\end{flushleft}
\begin{flushright}
 $ K'(\matP\left[d_H(p,X_np)\geq (2i+1)K'\right]-\matP\left[d_H(p,X_np)\geq (2i+3)K'\right]).$
\end{flushright}

For later purposes, we fix distinct $H_1,H_2\in\calH$. For $w$ a word in the elements of the support of $\mu$, denote by $g(w)$ the corresponding element of $g$. Let $w$ be a word of length $n$ so that $d_H(p,g(w)p)\geq l+100K_1$, for some large enough $K_1$. As $\pi_H$ is coarsely Lipschitz, the minimal subword $w_1$ of $w$ so that $d_H(p,g(w_1)p)\geq l$ actually satisfies $d_H(p,g(w_1)p)\in [l,l+K_1]$. Let $w_2$ be the subword of $w$ starting right after $w_1$, of length $K_2$ (where $K_2$ is a large enough constant) and let $w_3$ be the final subword of $w$ starting after $w_2$.

\begin{figure}[h]
 \includegraphics[scale=0.6]{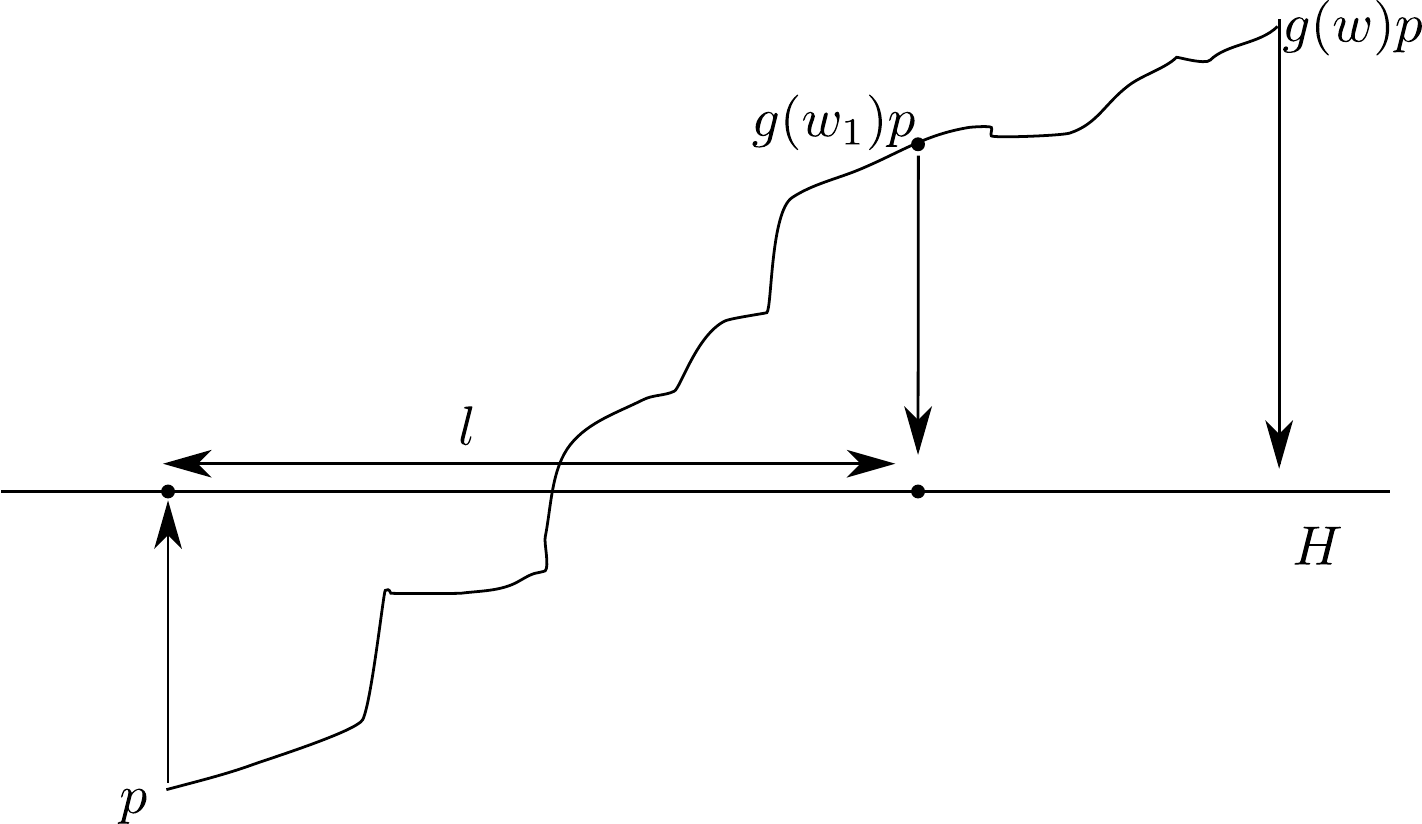}
\end{figure}

We claim that substituting $w_2$ by a suitable word $w_2'$ of the same length we can make sure that $d_{g(w_1) H_i}(H, g(w_1w'_2w_3)p)$ is larger than $B$ as in Corollary \ref{B-type}, where $i$ is chosen so that $g(w_1)H_i\neq H$. (The word $w'_2$ depends on $w$.)

Indeed, if $w'_2$ represents an element in the stabiliser of $H_i$ then (up to bounded additive error)
$$d_{g(w_1) H_i}(H,g(w_1w'_2w_3)p)=d(\pi_{H_i}(g(w_1)^{-1}H),g(w'_2)\pi_{H_i}(g(w_3)p)),$$
so that we can just ``push'' $\pi_{H_i}(g(w_3)p)$ far from $\pi_{H_i}(g(w_1)^{-1}H)$ using $g(w'_2)$ if they happen to be close, and choose $w'_2$ so that $g(w'_2)$ is at distance at most $1$ from the identity otherwise. (The hypotheses that $H_i$ has infinite diameter tells us that there is ``enough space'' to do so.) 

\begin{figure}[h]
 \includegraphics[scale=0.6]{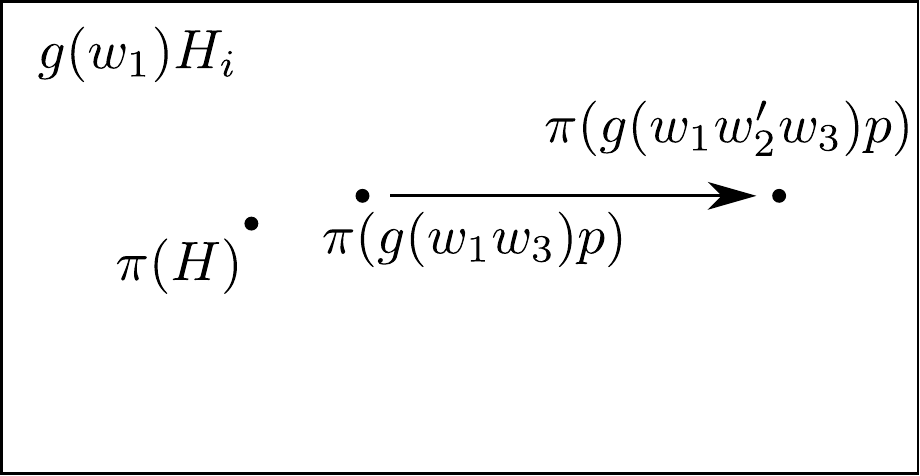}
\caption{$\pi$ denotes $\pi_{g(w_1)H}$.}
\end{figure}

In particular, keeping into account that $d(g(w_1)p,g(w_1)H_i)$ and hence $d_H(g(w_1)p,g(w_1)H_i)$ is bounded, we see that $d_H(g(w_1w'_2w_3)p,g(w_1)p)$ is bounded by Corollary \ref{B-type}. Hence $d_H(p,g(w_1w'_2w_3)p)\in [l-K_3,l+K_3)$, for a suitable $K_3$.

To sum up, we constructed for any word $w$ of length $n$ so that $d_H(p,g(w)p)\geq l+100K_1$ another word $w'=w_1w_2'w_3$ of length $n$ so that $d_H(p,g(w')p)\in [l-K_3,l+K_3)$, and the map $w\mapsto w'$ is easily seen to be bounded-to-1, as the decomposition of $w'$ as $w_1w'_2w_3$ is uniquely determined by the definition of $w_1$ and the length of $w'_2$, which is some fixed constant. This then gives the desired inequality as the support of $\mu$ is finite and hence for each $s,s'$ in the support $\mu(s)/\mu(s')$ is uniformly bounded.
\par
$(2)$ For $K$ large, we can assign to each $H\in A(g)$ a subgeodesic $\gamma_H$ of any given geodesic $\gamma$ from $p$ to $gp$ so that $l(\gamma_H)\geq 1$ and distinct $\gamma_H$'s are disjoint, such subgeodesic being just the deep component along $H$, see Lemma \ref{transprop}-(2)-(4).
\end{proof}

Finally, we show that subwalks of at least logarithmic length are expected to make linear progress (the definition of linear progress is above Theorem \ref{logtrk}). The lemma will also be used later.

\begin{lemma}
\label{linsubprog}
Let $G$ be a group acting on the metric space $X$ and suppose that the random walk $\{X_n\}$ on $G$ makes linear progress on $X$.
 Then for each $p\in X$ there exists $C_3$ so that for each $n\geq 1$
 $$\matP\left[\exists i,j\leq n: |i-j|\geq C_3\log (n), d(X_ip,X_{j}p)\leq |i-j|/C_3\right]\, \leq\, C_3 n^{-k}.$$
\end{lemma}

\emph{Proof.} As $\{X_n\}$ makes linear progress there exists $K_1$ so that
 $$\matP\big[d(p,X_jp)\leq j/K_1 \big]\leq K_1 e^{-j/K_1}.$$
 Notice that for each $i<j$ we have $\matP[d(X_ip,X_{j}p)\leq t]=\matP[d(p,X_i^{-1}X_{j}p)\leq t]=\matP[d(p,X_{j-i}p)\leq t]$. Summing the linear progress inequality for $i$ ranging from $1$ to $n$ and $j\geq i+(k+1)K_1\log n$ we get that the probability in the statement is at most
\[
 \pushQED{\qed}
\sum_{i=1}^n \sum_{l\geq (k+1)K_1 \log(n)} K_1 e^{-l/K_1}= n \sum_{i\geq 0} K_1 n^{-(k+1)}e^{-i/K_1}.\qedhere
 \popQED
\]

\subsection{Proof of Theorem \ref{logtrk}} The statement for $[p,X_np]$ follows from the one for $trans([p,X_np])$ in view of Lemma \ref{smallproj} and Lemma \ref{transprop}-(3) (or otherwise directly showing that the transient set is logarithmically dense in the corresponding geodesic). Fix $n\geq 2$, and denote by $C_i$ suitable constants that do not depend on $n$. Consider a sample path $\{w_i\}_{i\leq n}$ of the random walk. By Lemma \ref{div}, we have that each transient point on a geodesic from $p$ to $w_np$ is logarithmically close to $\{w_ip\}$, say at distance at most $C_4\log(n)$.
Now, assume that for all $H\in \calH$ we have $d_H(p,w_np)\leq C_2\log(n)$ and for all $i,j$ with $|i-j|\geq C_3\log(n)$ we have $d(w_ip,w_jp)\geq |i-j|/C_3$, where $C_2,C_3$ are as in Lemma \ref{smallproj} and Lemma \ref{linsubprog}. The lemmas tell us that we can safely disregard sample paths not satisfying these properties. Suppose that for some $j$ we have $d(w_jp, trans([p,w_np]))>C_4\log(n) $, so that any transient point on $[p,w_np]$ is $C_4\log(n)$-close to either $\{w_ip\}_{i<j}$ or $\{w_i\}_{i>j}$.
If $p'\in trans([p,w_np])$ is the closest point to $w_np$ satisfying $d(p',\{w_ip\}_{i<j})\leq C_4\log(n)$, then we also have $d(p', \{w_ip\}_{i>j})\leq C_5\log(n)$, as $trans([p,w_np])$ has at most logarithmic ``gaps''. Hence, there are $i_0<j<i_1$ so that $d(w_{i_0}p,w_{i_1}p)\leq (C_4+C_5)\log(n)=C_6\log(n)$.

\begin{figure}[h]
 \includegraphics[scale=0.7]{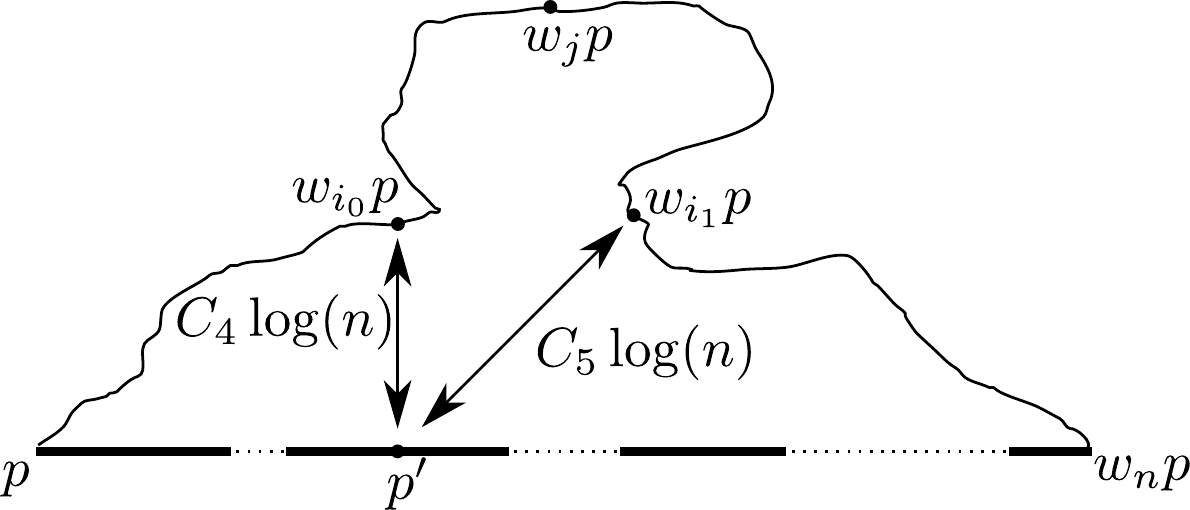}
\caption{$i_1-i_0$ cannot be large because $w_{i_0}p$ is close to $w_{i_1}p$.}
\end{figure}

By Lemma \ref{linsubprog} either $i_1-i_0< C_3\log(n)$ or $d(w_{i_0}p,w_{i_1}p) \geq (i_1-i_0)/C_3$, in which case $C_6\log(n)\geq d(w_{i_0}p,w_{i_1}p) \geq (i_1-i_0)/C_3$. Hence, in any case we have $i_1-i_0\leq C_7\log(n)$. Therefore,
\begin{align*}
 d(w_jp,trans([p,w_np]))&\leq d(w_jp,w_{i_0}p)+d(w_{i_0}p,trans([p,w_np]))\\
& \leq (C_8+C_4)\log(n),
\end{align*}
and this completes the proof.\qed

\subsection{Applications}
\label{appl}

Given three points $x_1,x_2,x_3$ in a geodesic metric space denote by $\delta(x_1,x_2,x_3)$ the supremum of the thinness constants of geodesic triangles with vertices $x_1,x_2,x_3$, i.e.
$$\sup \{d(p,[x_{i-1},x_i]\cup[x_{i+1},x_{i-1}])\},$$
where the supremum is taken over all choices of geodesics $[x_i,x_{i+1}]$ and all $p\in [x_i,x_{i+1}]$ for $i=1,2,3$ (we take indices modulo $3$).

\begin{defn}
 Let $\omega:\N\to\N$ be a function. We will say that \emph{random triangles in the group $G$ are $\omega$-thin} if whenever $\{X_n\},\{Y_n\},\{Z_n\}$ are independent simple random walks on $G$ with respect to the same generating system we have $\matE[\delta(X_n,Y_n,Z_n)]= O(\omega)$. Also, we say that \emph{the Gromov product of random points in $G$ is $\omega$-small} if whenever $\{X_n\},\{Y_n\}$ are independent simple random walks on $G$, again with respect to the same generating system, we have $\matE\left[\sup_{[X_n,Y_n]} d(1,[X_n,Y_n])\right]=O(\omega)$.
\end{defn}

A related notion called statistical hyperbolicity has been considered in \cite{DLM-stathyp,DDM-stathypteich}, where it is proven for hyperbolic groups and Teichm\"{u}ller spaces.

\begin{thm}
 Random triangles in any given non-trivial relatively hyperbolic group $G$ are $\log(n)$-thin. Also, the Gromov product of random points in $G$ is $\log(n)$-small.
\end{thm}

\begin{proof}
 As the random walks we are considering are independent and symmetric, we can concatenate them to obtain a longer random walk, meaning that the distribution of $X_j^{-1}Y_k$ is the same as that of $X_{j+k}$, and similarly for the other pairs.
 By Theorem \ref{logtrk} and the observation above  we have an appropriate constant $C$ so that
$$\matP[d_{Haus}([X_n,Y_n],\{X_i,Y_i\}_{i\leq n})\geq C\log(n)]=$$
$$\matP[d_{Haus}([1,X_n^{-1}Y_n],\{X^{-1}_{n}X_i,X^{-1}_nY_i\}_{i\leq n})\geq C\log(n)]$$
goes to $0$ faster than, say, $1/n^2$, and similarly for the other pairs. Both conclusions easily follow.
\end{proof}

\begin{rem}
 Notice that a similar statement holds for random polygons as well. Also, the theorem can be generalised to group actions on relatively hyperbolic spaces. 
\end{rem}

Recall that given a group $G$ and a generating system for $G$, a (discrete) combing $\Gamma$ is a choice, for each $x\in G$, of a discrete path $\Gamma(x)$ connecting $1$ to $x$ in the Cayley graph of $G$. Given a finitely presented group $G$ and a (discrete) loop $\alpha$ in its Cayley graph, we will denote by $Fill(\alpha)$ the area of a minimal van Kampen diagram whose boundary is $\alpha$.

\begin{defn}
 Let $\Gamma$ be a combing on the finitely presented group $G$ and $\{X_i\}$ a simple random walk on $G$. The \emph{average Dehn function} $\delta_{avg}^{G,\Gamma, \{X_i\}}$ of $G$ with respect to $\Gamma$ and $\{X_i\}$ is
 $$\delta_{avg}^{G,\Gamma, \{X_i\}}(n)=\matE[Fill(\{X_0,\dots,X_n\}\cup \Gamma(X_n))].$$
\end{defn}

In \cite{Y-avdehnnilp} another notion of average Dehn function is considered and it is shown that for most nilpotent groups this function is subasymptotic to the Dehn function. Other related results can be found in \cite{KMS-avcompl,BV-meandehn}.

\begin{thm}
 Suppose that $G$ is hyperbolic relative to proper subgroups with at most polynomial Dehn function. Then for each geodesic combing $\Gamma$ and every simple random walk $\{X_i\}$ on $G$ we have
 $$\delta_{avg}^{G,\Gamma, \{X_i\}}(n)=O(\, n\, \delta(\log (n))\, ),$$
 where $\delta$ is the maximum of the Dehn functions of the peripheral subgroups.
\end{thm}

\begin{proof}
 Fix $n\geq 2$. By Theorem \ref{logtrk}, we can take the expected value that defines $\delta_{avg}$ conditioned on $d_{Haus}(trans(\Gamma(X_n)),\{X_i\}_{i\leq n})\leq C\log(n)$ for some appropriate $C$. We are allowed to do so because the Dehn function of $G$ is (equivalent to) $\delta$ \cite{Fa-relhyp}, so that by choosing $C$ large enough we can make sure that $\matP[d_{Haus}(trans(\Gamma(X_n)),\{X_i\}_{i\leq n})> C\log(n)]$ decays much faster than the inverse of the Dehn function of $G$ (informally speaking, the loops we are disregarding are too few to contribute to the average Dehn function). Choose discrete geodesics $\alpha_i$ connecting $X_i$ to $trans(\Gamma(X_n))$ of length at most $C\log(n)$ (choose $\alpha_0$ and $\alpha_n$ to be trivial), and consider discrete loops $l_i$ obtained concatenating $\alpha_i$, a subgeodesic of $\Gamma(X_n)$ and $\alpha_{i+1}^{-1}$. Each of these loops has area at most $K\delta(\log(n))$ for some suitable $K$. As there are $n-1$ such loops, and from fillings of all of them we can recover a filling of the full path, we get the desired bound.
\end{proof}

It would be interesting to know whether the same result holds for the notion of average Dehn function defined in terms of the uniform distribution on loops.

\section{Mapping class groups}

\begin{conv}
\label{conv}
 From now on, all surfaces will be assumed to be orientable, connected, of finite type and to have empty boundary.
\end{conv}

Recall that the complexity of a surface $S$ is $3g+p-3$, where $g$ is the genus of $S$ and $p$ the number of punctures. In particular a surface has complexity at least 2 if it is not a sphere with at most four punctures or a torus with at most one puncture.
In this section we show the following. 

\begin{thm}
\label{trkmcg}
 Let $G$ be the mapping class group of a surface of complexity at least 2, let $\mu$ be a finitely supported probability measure on $G$ and let $\{X_n\}$ be the corresponding random walk on $G$. Then
$$\matE\left[\sup_{\gamma(X_n)} d_{Haus}(\{X_i\}_{i\leq n}, \gamma(X_n))\right] =O(\sqrt{n\log(n)}),$$
where the supremum is taken over all hierarchy paths and geodesics $\gamma(X_n)$ in a given word metric from $1$ to $X_n$.

More precisely, for each $k$ there exists $C$ so that for all $n\geq 1$
$$\matP\left[sup_{\gamma(X_n)} d_{Haus}(\{X_i\}_{i\leq n}, \gamma(X_n))\geq C\sqrt{n\log(n)}\right] \leq Cn^{-k}.$$
\end{thm}

\subsection{Results from the literature}
We will assume that the reader is somewhat familiar with hierarchies, subsurface projections and related notions from \cite{MM1,MM2}, but we recall the main results we will need. We denote by $S$ a surface of complexity at least 2, and let $\calM(S)$ be its mapping class group (which for our purposes can be identified with the marking complex). With an abuse, when referring to a subsurface we will actually refer to its isotopy class and assume that it is connected and essential. For $Y\subseteq S$ a subsurface, $\pi_Y:\calM(S)\to 2^{\calC(Y)}$ will denote the subsurface projection on the curve complex $\calC(Y)$ of $Y$, which is hyperbolic \cite{MM1} (see \cite{HPW-unic} for a short and self-contained proof of this fact). Recall that $\pi_Y$ is coarsely Lipschitz \cite{MM1}. As customary, we also denote $\pi_Y$ the subsurface projection as a map from $\calC(S)$ to $2^{\calC(Y)}$. The surfaces $Y,Z$ are said to overlap if $\pi_Y(\partial Z), \pi_Z(\partial Y)\neq \emptyset$. We use the notation $d_Y(\mu,\nu)=diam_{\calC(Y)}(\pi_Y(\mu)\cup\pi_Y(\nu))$.

\begin{thm}[Behrstock Inequality, {\cite[Theorem 4.3]{Be-th}}]
\label{Behin}
There exists a constant $C$ so that if the subsurfaces $Y, Z\subseteq S$ overlap then for each $m\in\calM(S)$ we have
$$\min\{d_Y(\partial Z, m), d_Z(\partial Y,m)\}\leq C.$$
\end{thm}

(The constant can be chosen to be $10$ in view of a slick argument due to Leininger and written up in \cite{Man-unexpmcg}.)

We write $A\approx_{K,C} B$ if the quantities $A,B$ satisfy
$$A/K-C\leq B\leq KA+C.$$
Let $\{\{A\}\}_L$ denote $A$ if $A\geq L$ and $0$ otherwise.

\begin{thm}[Distance Formula, {\cite[Theorem 6.12]{MM2}}]
There exists $L_0$ with the property that for each $L\geq L_0$ there are $K,C$ so that, for each $\mu,\nu\in \calM(S)$,
 $$d_{\calM(S)}(\mu,\nu)\approx_{K,C} \sum_Y \{\{d_Y(\mu,\nu)\}\}_L,$$
where the sum is taken over all (isotopy classes of) subsurfaces $Y$.
\end{thm}

For notational convenience, from now on we denote sums over all subsurfaces of a surface $S$ simply by $\sum_{Y\subseteq S}$.

\begin{thm}[Bounded Geodesic Image Theorem, {\cite[Theorem 3.1]{MM2}}]
 There exists $C$ with the following property. If $\gamma$ is a geodesic in $\calC(S)$ so that for some proper subsurface $Y$ we have that $\pi_Y(v)$ is non-empty for every vertex $v\in\gamma$ then $\pi_Y(\gamma)$ has diameter at most $C$.
\end{thm}

An elementary proof of this theorem can be found in \cite{We-bgi}.

\begin{rem}
\label{bgi}
 The consequence of the theorem we will more often use is that if $Y$ has complexity $\xi(S)-1$ and $\pi_Y(\gamma)$ has sufficiently large diameter then a component of $\partial Y$ appears in $\gamma$. More generally, if $\pi_Y(\gamma)$ has sufficiently large diameter then $Y$ is contained in $S\backslash v$ for some $v\in\gamma$ or it is an annulus around some such $v$.
\end{rem}

We also recall the following result on random walks due to Maher (not stated in full generality).

\begin{thm}\cite{Ma-linprog}\cite[Theorem 1.2]{Ma-expdecay}
\label{linprog}
Random walks on a mapping class group $\calM(S)$ of surfaces of complexity at least 2 whose finite support generates $\calM(S)$ make linear progress in the curve complex.
\end{thm}

\subsection{Quantitative quadratic divergence}
\label{divsec}

Let $L_0$ be as in the Distance Formula and larger than the constant $C$ in the Bounded Geodesic Image Theorem. Similarly to \cite{Be-th}, we say that a pair $x,y\in \calM(S)$ is $D$-bounded if for each proper subsurface $Y$ of $S$ we have $\sum_{Z\subseteq Y} \{\{d_Z(x,y)\}\}_{3L_0}\leq D$.

We set the threshold (almost) arbitrarily, but for our purposes different thresholds give equivalent notions of boundedness, as we can see from the following lemma (a straightforward consequence of the Distance Formula in the case $Y=S$).

\begin{lemma}
\label{threshold}
 For each $L\geq L_0$ there exists $C$ so that for each $x,y\in\calM(S)$ and $Y\subseteq S$ we have
$$\sum_{Z\subseteq Y} \{\{d_Z(x,y)\}\}_L\leq \sum_{Z\subseteq Y} \{\{d_Z(x,y)\}\}_{L_0}\leq C \sum_{Z\subseteq Y} \{\{d_Z(x,y)\}\}_L+C.$$
\end{lemma}

This fact is implicit in the proof of the Distance Formula (as are a few facts that will appear in the proofs below). However, in order to make the proofs we give accessible to more readers, we will rely as little as possible on the machinery of hierarchies and use the Distance Formula instead.

\begin{proof}
 The first inequality is obvious. In order to show the second one we would like to bound the number of subsurfaces $Z\subseteq Y$ so that $d_Z(x,y)\in [L_0,L)$. We proceed inductively on complexity. There is at most one such $Z$ of complexity $\xi(Z)=\xi(Y)$, i.e. $Z=Y$. Suppose we are given inductively subsurfaces $\calZ_k=\{Z_1,\dots Z_{i(k)},Z'_1,\dots, Z'_{j(k)}\}$ of complexity $\xi(Y)-k$ so that $d_{Z_i}(x,y)< L$, $d_{Z'_j}(x,y)\geq L$ and so that any subsurface $Z$ of positive complexity at most $\xi(Y)-k$ with $d_Z(x,y)\geq L_0$ is contained in some $Z_i$ or in some $Z'_i$.
Let $\hat{Z}\in \calZ_k$. Choose a geodesic $\gamma$ from $\pi_{\hat{Z}}(x)$ to $\pi_{\hat{Z}}(y)$ and choose at most $2d_{\hat{Z}}(x,y)+2$ subsurfaces $\{Z''_i\}$ so that for every vertex $v\in\gamma$ any component of $\hat{Z}\backslash v$ is contained in one such subsurface. By the Bounded Geodesic Image Theorem (see also Remark \ref{bgi}), any subsurface $Z'$ of positive complexity lower than that of $Y$ and such that $d_{Z'}(x,y)\geq L_0$ is contained in some $Z''_i$. We then get a bound of the form
$$|i(k+1)|\leq |i(k)|(2L+2)+4\sum_j d_{Z'_j}(x,y),$$
where for convenience we used that $d_{Z'_j}(x,y)\geq 1$ and hence $2d_{Z'_j}(x,y)+2\leq 4d_{Z'_j}(x,y)$.
It is also easy to bound the number of annuli $A\subseteq Y$ with $d_A(x,y)\in[L_0,L)$, again using Remark \ref{bgi}.
Proceeding inductively one shows that the number of subsurfaces $Z\subseteq Y$ so that $d_Z(x,y)\in [L_0,L)$ is at most
$$p(L)+p(L)\sum_{Z\subseteq Y} \{\{d_Z(x,y)\}\}_L,$$
for an appropriate polynomial $p(x)$ (notice that each $Z$ contributing a non-zero term to the sum appears at most once in the inductive procedure).
\end{proof}

For $x,y\in\calM(S)$, we denote by $[x,y]$ a hierarchy path joining them.

\begin{prop}
\label{divmcg}
Let $S$ be a surface of complexity at least 2.
 There is a constant $C=C(S)$ so that the following holds. Let $D\geq 1$ and let $x_1,x_2\in \calM(S)$ be a $D$-bounded pair. Then for any path $\alpha$ of length at least $1$ from $x_1$ to $x_2$ and any $p\in [x_1,x_2]$ we have
$$d(p,\alpha)\leq C \sqrt{D l(\alpha)}.$$
\end{prop}

The argument below can be somewhat simplified if one only wants to reprove quadratic divergence, as in this case one can use $D$ as a threshold in the distance formula. We cannot do this because the error terms in the distance formula are not linear in the threshold.

\begin{proof}
We will denote by $C_i$ suitable large enough constants, depending on $S$ only. We fix $L_0$ as above and set $L=3L_0$.

Let $x_1,x_2,\alpha$ be as in the statement and let $p\in[x_1,x_2]$ be a point maximising the distance from $\alpha$. Set $d=d(p,\alpha)$, which we can safely assume to be larger than $1000\delta$, where $\delta$ is the maximum of the hyperbolicity constants of all curve complexes of subsurfaces of $S$. Consider a subpath $[x'_1,x'_2]$ of $[x_1,x_2]$ with the property that $d(x'_i,p)\in [d/2,3d/4]$, so that in particular $d(x'_i,\alpha)\in [d/4,d]$. Also, consider $p_i\in\alpha$ so that $d(x'_i,p_i)\leq d$. Denote $\alpha'$ the subpath of $\alpha$ from $p_1$ to $p_2$. Pick a maximal collection of points $y_1,\dots, y_n$ on $[x'_1,x'_2]$ so that $d_S(y_i,y_{i+1})\geq C_0\geq 100\delta$.
Notice that any pair of points $p,q$ on $[x_1,x_2]$ is $C_1D$ bounded because for each $Y\subseteq S$ we have $d_Y(p,q)\leq d_Y(x,y)+C_2$ (a standard property of hierarchy paths), and in view of Lemma \ref{threshold}.

 We claim that we have 
$$n\geq d/(C_3D).$$
This follows from the lemma below, which we record for later purposes.

\begin{lemma}
\label{fewlargedom}
 There exists $C$ with the following property. Let $p,q\in \calM(S)$ be an $E$-bounded pair for some $E\geq 1$. 
Then
$$d_{\calM(S)}(p,q)\leq C E d_S(p,q).$$ 
\end{lemma}

\emph{Proof.} By the Bounded Geodesic Image Theorem (see also Remark \ref{bgi}) one can find $C_4d_S(p,q)$ subsurfaces $Y_i\subsetneq S$ so that any subsurface $Y\subsetneq S$ with $d_Y(p,q)>L$ is contained in some $Y_i$, as for each such subsurface (a component of) $\partial Y_i$ has to be contained in any given geodesic from $\pi_S(p)$ to $\pi_S(q)$. Using the Distance Formula we get
\[
 \pushQED{\qed}
d_{\calM(S)}(p,q)\leq C_5 d_S(p,q)+C_5\sum_i\sum_{Y\subseteq Y_i} \{\{d_Y(p,q)\}\}_{L}\leq C_6d_S(p,q) E.\qedhere
 \popQED
\]

There are points $q_1,\dots,q_n$ on $\alpha$ so that the closest point projection of $\pi_S(q_i)$ on a geodesic $\gamma$ in $\calC(S)$ from $\pi_{S}(x'_1)$ to $\pi_S(x'_2)$ is within bounded error $\pi_S(y_i)$ (and which appear in the given order along $\alpha$).

\begin{figure}[h]
 \includegraphics[scale=0.7]{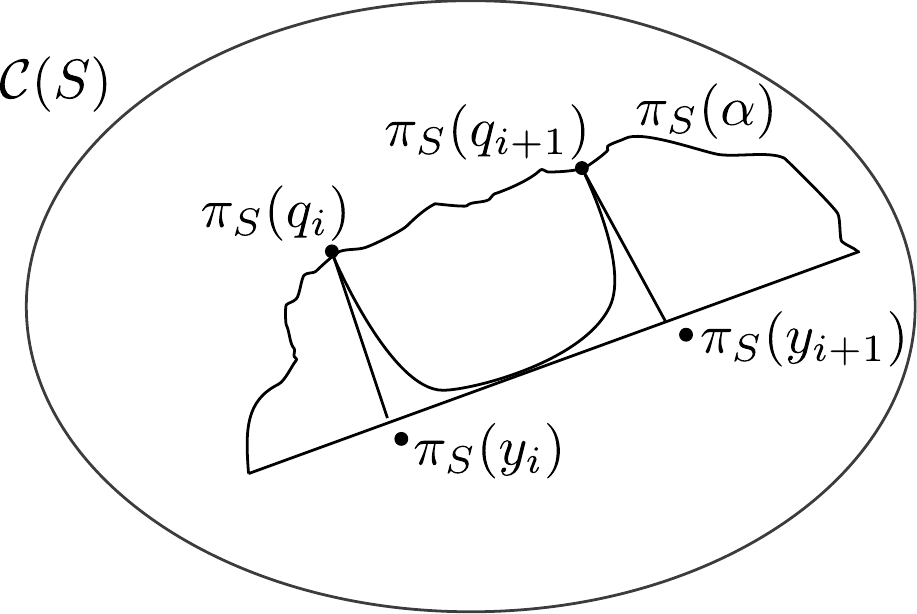}
\caption{The situation as seen from the curve complex.}
\end{figure}

We will now show that $d_{\calM(S)}(q_i,q_{i+1})$ can be bounded from below linearly in $d$. At the end of the proof we will just combine such lower bound with the lower bound on $n$.

Let $Y_1,\dots, Y_k$ be subsurfaces so that
$$\sum_j \{\{d_{Y_j}(q_i,y_i)\}\}_{3L} \geq d/C_7,$$
and all terms in the sum are positive. Such subsurfaces exist by the Distance Formula, as $d_{\calM(S)}(q_i,y_i)\geq d/4$. (Notice that we set the threshold to $3L$.)
By Lemma \ref{fewlargedom}, we know that $d_{\calM(S)}(y_i,y_{i+1})$ is bounded linearly in $D$.

Once again by the Distance Formula we have
$$\sum_j\{\{d_{Y_j}(y_i,y_{i+1})\}\}_{L}\leq C_{8}D.$$
Notice that if $l(\alpha)\leq \sqrt{D}$ then $d(p,x_1)$ is $O(\sqrt{D})$, so that we can assume $l(\alpha)\geq \sqrt{D}$. Hence, we can also assume that $d$ is larger than, say, $100C_7C_{8}D$. We then claim that
$$\sum_j \{\{d_{Y_j}(q_i,y_{i+1})\}\}_{2L} \geq d/C_{9}.$$
(Notice that we lowered the threshold.)
This follows combining the facts that for each $j$ we have
$$\{\{d_{Y_j}(q_i,y_{i+1})\}\}_{2L}\geq \{\{d_{Y_j}(q_i,y_{i})\}\}_{3L}-\{\{d_{Y_j}(y_i,y_{i+1})\}\}_{L}-2L$$
and
$$\{\{d_{Y_j}(q_i,y_{i})\}\}_{3L}-2L\geq \frac{1}{3} \{\{d_{Y_j}(q_i,y_{i})\}\}_{3L},$$
as the threshold is attained for each $j$ by hypothesis.

By hyperbolicity of $\calC(S)$, geodesics from $\pi_S(q_{i+1})$ to $\pi_S(y_{i+1})$ stay far from geodesics from $\pi_S(q_i)$ to $\pi_S(y_i)$, and hence $d_{Y_j}(q_{i+1},y_{i+1})$ can be bounded by the Bounded Geodesic Image Theorem for $Y_j\subsetneq S$. On the other hand, if for some $j$ we have $Y_j=S$ then $d_{Y_j}(q_i,y_{i+1})\leq d_{Y_j}(q_i,q_{i+1})/C_{10}$ because any geodesic from $\pi_S(q_i)$ to $\pi_S(q_{i+1})$ passes close to $\pi_S(y_i), \pi_S(y_{i+1})$. In particular 
$$\sum_j \{\{d_{Y_j}(q_i,q_{i+1})\}\}_{L} \geq d/C_{11}.$$
and we can therefore use the Distance Formula to give the lower bound
$$d_{\calM(S)}(q_{i},q_{i+1})\geq d/C_{12}.$$
Thus, we get
$$l(\alpha)\geq \sum d(q_i,q_{i+1})\geq (n-1)d/C_{12}\geq d^2/(C_{13} D),$$
the inequality we were looking for.
\end{proof}

A standard argument now gives the following.

\begin{cor}
\label{Morse}
 Fix the notation of Proposition \ref{divmcg} and assume furthermore that $\alpha$ is a geodesic in a given word metric. Then (up to increasing $C$)
$$d_{Haus}(\alpha,[x_1,x_2])\leq C\sqrt{Dd(x_1,x_2)}.$$
In particular, Proposition \ref{divmcg} still holds if $[x_1,x_2]$ denotes a geodesic in a given word metric rather than a hierarchy path.
\end{cor}

\begin{proof}
 The fact that any point on $[x_1,x_2]$ is close to $\alpha$ is the content of the proposition. Set $A=C\sqrt{Dd(x_1,x_2)}$. Any $q\in\alpha$ splits $\alpha$ into two subgeodesics $\alpha_1,\alpha_2$. All points on $[x_1,x_2]$ are $A$-close to either $\alpha_1$ or $\alpha_2$, so that there is a point $p\in[x_1,x_2]$ which is $A$-close to both (this is true up to bounded error if $[x_1,x_2]$ is regarded as a discrete path). Hence, $q$ is contained in a subgeodesic of length at most $2A$ whose endpoints are $A$-close to $[x_1,x_2]$. In particular, $q$ is $2A$-close to $[x_1,x_2]$.
\end{proof}

\subsection{Proof of Theorem \ref{trkmcg}} Fix from now on the notation of Theorem \ref{trkmcg}. We now show that pairs $1,X_n$ are expected to be $O(\log(n))$-bounded.
\begin{lemma}
\label{smallprojmcg}
 For each $k\geq 1$ there exists $C_0$ so that, for each $n\geq 2$,
$$\matP\left[1,X_n\mathrm{\ is\ not\ }C_0\log(n)\mathrm{-bounded} \right]\leq C_0n^{-k}.$$
\end{lemma}

\emph{Proof.} Fix $L_0$ as in the distance formula and larger than the constant $C$ in the Bounded Geodesic Image Theorem. For $Y\subsetneq S$ and $x,y\in \calM(S)$ we denote $\sigma_Y(x,y)=\sum_{Y'\subseteq Y} \{\{d_{Y'}(x,y)\}\}_{3L_0}$ the contribution made to the Distance Formula by subsurfaces of $Y$. Notice that $\sigma_Y$ is a coarsely Lipschitz function (in both variables). 

Let $\{w_i\}_{i\leq n}$ be a sample path. Consider a geodesic $\gamma$ from $\pi_S(1)$ to $\pi_S(w_n)$. By Theorem \ref{logtrk}, we can assume $d_{Haus}(\gamma,\{\pi_S(w_i)\}_{i\leq n})\leq K_1\log(n)$ (in $\calC(S)$). Also, we can assume that for each $i,j\leq n$ with $|i-j|\geq K_1\log(n)$ we have $d_S(w_i,w_j)\geq |i-j|/K_1$ by Lemma~\ref{linsubprog}. Both results apply in view of the linear progress in $\calC(S)$, Theorem \ref{linprog}.

Let $Y$ be any subsurface. If $d_S(\gamma, \partial Y)$ is positive then the projection of $\gamma$ on any subsurface of $Y$ is bounded by $L_0$, so assume that this is not the case. The idea is to split the random path into an initial, central and final part. The initial and final part will make bounded contribution to $\sigma_{Y}$ in view of the Bounded Geodesic Image Theorem, while the contribution of the central part can be estimated using that $\sigma_{Y}$ is coarsely Lipschitz.

 Let $p\in\gamma$ be so that $d_S(p,\partial Y)\leq 10$ and let $w_i$ be so that $d_S(p,w_i)\leq K_1\log(n)$. Consider initial and final subgeodesics $\gamma_1,\gamma_2$ of $\gamma$ at distance at least $K_1\log(n)+100$ from $p$. Then an initial subpath of the random path will be close to $\gamma_1$, while a final subpath will be close to $\gamma_2$, that it to say for $K_2$ large enough we have that if $|i-j|\geq K_2\log(n)$ then $d(\gamma_l,w_j)\leq K_1\log(n)$, where $l=1$ if $j<i$ and $l=2$ if $j>i$. The Bounded Geodesic Image Theorem implies that the diameter of the projection of the $\gamma_l$'s on any $Y'\subseteq Y$ is bounded, and likewise for all geodesics from $w_j$ to the closest point to $w_j$ in the corresponding $\gamma_l$, if $|i-j|\geq K_2\log(n)$. Thus, we can give a uniform bound on $diam(\pi_{Y'}(\{w_j\}_{0\leq j\leq i-K_2\log(n))})$ and $diam(\pi_{Y'}(\{w_j\}_{i+K_2\log(n)\leq j\leq n}))$ for any subsurface $Y'\subseteq Y$. Define $\sigma'_Y(x,y)=\sum_{Y'\subseteq Y} \{\{d_{Y'}(x,y)\}\}_{L_0}$, which is again coarsely Lipschitz. We can now make the estimate
 $$\sigma_Y(1,w_n)\leq \sigma'_Y\left(1,w_{i-K_2\log(n)}\right)+\sigma'_Y\left(w_{i-K_2\log(n)},w_{i+K_2\log(n)}\right)$$
\[
 \pushQED{\qed}
+\sigma'_Y\left(1,w_{i+K_2\log(n)}\right)\leq K_3\log(n).\qedhere
 \popQED
\]

\medskip

We can now conclude the proof of Theorem \ref{trkmcg}. Fix $n\geq 2$ (the constants $C_i$ below do not depend on $n$). Consider a sample path $\{w_i\}_{i\leq n}$ so that $1,w_n$ is $C_0\log(n)$-bounded and for each $i,j$ with $|i-j|\geq C_0\log(n)$ we have $d_S(w_i,w_j)\geq |i-j|/C_0$, see Lemma \ref{smallprojmcg} and Lemma \ref{linsubprog}. First of all, any point on $[1,w_n]$ has distance at most $C_1\sqrt{n\log(n)}$ from the sample path by Proposition \ref{divmcg}. Suppose that there is $w_i$ so that $d(w_i,[1,w_n])>C_1\sqrt{n\log(n)}$. Notice that there is a point $p$ on $[1,w_n]$ so that
$$d(p,\{w_j\}_{j<i}),d(p,\{w_j\}_{j>i})\leq C_1\sqrt{n\log(n)},$$
for example the first point along $[1,w_n]$ that is $C_1\sqrt{n\log(n)}$-close to $\{w_j\}_{j>i}$. In particular, there are $j_0<i<j_1$ so that $d(w_{j_0},w_{j_1})\leq 2C_1\sqrt{n\log(n)}$ and $d(w_{j_0},[1,w_n])\leq C_1\sqrt{n\log(n)}$. From Lemma \ref{linsubprog} one easily gets (as in the proof of Theorem \ref{logtrk}) that $j_1-j_0\leq C_2 \sqrt{n\log(n)}$. From this, it easily follows that $d(w_i,[1,w_n])\leq C_3 \sqrt{n\log(n)}$.

In view of Corollary \ref{Morse} and Lemma \ref{smallprojmcg}, the statement for geodesics follows from the one for hierarchy paths.

\section{Random projections estimate}
The reason why the proof of Lemma \ref{smallprojmcg} is different from that of Lemma \ref{smallproj} is just to illustrate two different techniques to get projection estimates for random points. We think it is worthwhile to state fact $(1)$ in Lemma \ref{smallproj} as well as its counterpart in the mapping class group as a separate lemma.

\begin{lemma}[Small Random Projections]
\label{smallrndprj}
Let $G$ be a non-trivial relatively hyperbolic group (resp. mapping class group of a surface of complexity at least 2). Consider a random walk generated by a symmetric probability measure whose finite support generates $G$. Then there exists $K$ with the following property. If $H$ is a left coset of a peripheral subgroup (resp. proper subsurface) then for all $l\geq 0$ and all $n\geq 1$ we have
  $$\matP\left[d(\pi_H(1),\pi_H(X_n))\geq l\right]\leq Ke^{-l/K}.$$
\end{lemma}

\begin{proof}
 We proved this fact within Lemma \ref{smallproj} for relatively hyperbolic groups. All the properties of projections on peripheral sets that we used in the proof have analogues for subsurface projections. In particular, the same proof almost goes through, see below, except that an extra argument is needed in the part where we used Corollary \ref{B-type}. In fact, the Behrstock Inequality for subsurface projections holds for \emph{overlapping} subsurfaces, so not for all pairs of disjoint subsurfaces. To solve this problem we can use the following fact \cite[Section 4.3]{BBF}. There exists a finite-index subgroup $G'$ of the mapping class group so that whenever two distinct subsurfaces are in the same $G'$-orbit they overlap. Notice also that there are finitely many $G'$-orbits.

In order to show the lemma, we can follow the proof of Lemma \ref{smallproj}(1) verbatim until the definition of $w_3$, where $\calH$ now denotes the collection of all proper subsurfaces of $S$ and we set $p=1$. The last part of the proof can be substituted by the following argument.

Choose two distinct subsurfaces $H^j_1,H^j_2$ from each $G'$-orbit. We claim that substituting $w_2$ by a suitable word $w'_2=u_2v_2$ of the same length we can make sure that $d_{g(w_1u_2) H^j_i}(\partial H, g(w_1w'_2w_3))$ is larger than $C$ as in the Behrstock Inequality, where 
\begin{enumerate}
 \item $u_2$ is chosen so that $w_1u_2$ represents an element of $G'$,
 \item $j$ is chosen so that $H^j_1,H^j_2$ are in the same $G'$-orbit as $H$,
 \item $i\in\{1,2\}$ is chosen so that $g(w_1u_2) H^j_i\neq H$.
\end{enumerate}

Notice that $u_2$ can be chosen from a finite list of words as $G'$ has finite index in the mapping class group, so that there is a uniform bound on the distance of $g(u_2)$ from $1$ and hence from $g(w_1)$ to $g(w_1u_2)$.
If $v_2$ represents an element in the stabiliser of $H_i$ then (up to an additive constant)
$$d_{g(w_1u_2) H_i}(\partial H,g(w_1w'_2w_3))=d(\pi_{H_i}(g(w_1u_2)^{-1}\partial H),g(v_2)\pi_{H_i}(g(w_3))),$$
so that for an appropriate choice of $v_2$ we have that $\pi_{H_i}(g(w_3))$ is far from $\pi_{H_i}(g(w_1u_2)^{-1}\partial H)$.

In particular, $d_H(g(w_1w'_2w_3),g(w_1u_2) \partial H^j_i)$ is bounded by the Behrstock Inequality. Also, $d=d_H(g(w_1u_2) \partial H^j_i, g(w_1u_2))$ can be uniformly bounded as $\pi_H(\partial H^j_i)$ coincides with the projection of a fixed marking $m^j_i$ whose set of base curves contains $\partial H^j_i$. As $\pi_H$ is coarsely Lipschitz and there are finitely many $H^j_i$'s, we can give a bound $d$. To sum up, there is a uniform bound on $d_H(g(w_1w'_2w_3),g(w_1u_2))$ and hence on $d_H(g(w_1w'_2w_3),g(w_1))$ as the word $u_2$ has bounded length. Hence $d_H(1,g(w_1w'_2w_3))\in [l-K_3,l+K_3)$, for a suitable $K_3$.

So, we constructed for any word $w$ of length $n$ so that $d_H(1,g(w))\geq l$ another word $w'=w_1w_2'w_3$ of length $n$ so that $d_H(1,g(w'))\in [l-K_3,l+K_3)$, and the map $w\mapsto w'$ is easily seen to be bounded-to-1. This then gives the desired inequality as the support of $\mu$ is finite and hence for each $s,s'$ in the support $\mu(s)/\mu(s')$ is uniformly bounded.
\end{proof}

\bibliographystyle{alpha}
\bibliography{rw}

\end{document}